\documentclass[11pt]{article}

\usepackage[T1]{fontenc}
\usepackage{amsmath,amssymb,amsthm}
\usepackage{graphicx}
\usepackage[colorlinks=true]{hyperref}

\theoremstyle{plain}
\newtheorem{theorem}{Theorem}[section]
\newtheorem{lemma}[theorem]{Lemma}
\newtheorem{corollary}[theorem]{Corollary}

\theoremstyle{remark}
\newtheorem{remark}[theorem]{Remark}

\DeclareMathOperator{\per}{per}
\DeclareMathOperator{\tr}{tr}

\newcommand{\ii}{\mathrm{i}}
\newcommand{\E}{\mathcal{E}}                 
\newcommand{\Eper}{\mathcal{E}_{\mathrm{per}}} 

\begin{document}

\title{Permanental Energy of Graphs}

\author{Priyanshu Pant\\
Indian Institute of Technology Indore\\
\texttt{priyanshupant03@gmail.com}
\and
Ranveer Singh\\
Indian Institute of Technology Indore\\
\texttt{ranveer@iiti.ac.in}}
\date{February 27, 2026}

\maketitle
\begin{abstract}
For a simple graph $G$ with adjacency matrix $A(G)$, let $\pi(G,x):=\per(xI-A(G))$ be its permanental
polynomial with roots $\mu_1,\dots,\mu_n\in\mathbb{C}$, and define the \emph{permanental energy}
$\Eper(G):=\sum_{i=1}^n|\mu_i|$.
We prove a sharp universal lower bound: for every $m$-edge graph $G$,
$\Eper(G)\ge 2\sqrt m$, with equality if and only if $G$ is a star together with isolated vertices.
We also prove the general upper bound $\Eper(G)\le n\rho(G)$, where $\rho(G)$ is the spectral radius,
and we study $\Eper(G)$ on several graph families.
\end{abstract}


\section{Introduction}
\label{sec:intro}
Let $G$ be a finite simple graph on $n$ vertices, with adjacency matrix $A(G)$ and eigenvalues $\lambda_1,\dots,\lambda_n$.
The \emph{energy} of $G$ is defined by
\[
\E(G)=\sum_{i=1}^n |\lambda_i|.
\]
Introduced by Gutman~\cite{Gutman1978}, $\E(G)$ has become a central invariant in spectral graph
theory and mathematical chemistry.
Its original motivation comes from the H\"uckel molecular orbital (HMO) model, where eigenvalues
of molecular graphs are used to approximate the total $\pi$-electron energy of conjugated
hydrocarbons; see~\cite{LiShiGutman2012}.
More generally, spectral quantities such as $\E(G)$ provide concise descriptors of graph structure;
see~\cite{CvetkovicDoobSachs1980} for background on graph spectra.

A bridge from linear algebra to graph invariants is via the \emph{characteristic polynomial}
\[
\chi(G,x)=\det(xI-A(G)),
\]
whose roots are exactly the eigenvalues $\lambda_1,\dots,\lambda_n$.
Since the determinant can be computed in $O(n^{\omega})$ arithmetic operations, where $\omega$ is the matrix-multiplication exponent~\cite{BunchHopcroft1974}, this bridge is useful not only conceptually but also computationally.
In particular, invariants such as $\E(G)$ have a rich theory and admit efficient computation.

The permanent looks deceptively similar to the determinant.
For an $n\times n$ matrix $M=(m_{ij})$,
\[
\det(M)=\sum_{\sigma\in S_n}\operatorname{sgn}(\sigma)\prod_{i=1}^n m_{i,\sigma(i)},
\qquad
\per(M)=\sum_{\sigma\in S_n}\prod_{i=1}^n m_{i,\sigma(i)},
\]
where $S_n$ is the symmetric group on $[n]:=\{1,2,\dots,n\}$ and $\operatorname{sgn}(\sigma)\in\{+1,-1\}$ denotes the sign of the permutation $\sigma$.
The only difference is the sign factor, yet the consequences are dramatic.
Valiant proved that computing $\per(M)$ is $\#P$-complete, even for $0$--$1$ matrices~\cite{valiant1979complexity},
making the permanent a central object in theoretical computer science.

The P\'olya Permanent Problem~\cite{polya1913aufgabe} asks when the permanent of a matrix can be transformed
into the determinant of a related matrix by suitable sign changes.
McCuaig~\cite{mccuaig2004polya} showed that this problem is equivalent to several combinatorial and
graph-theoretic problems, including counting perfect matchings in a graph.
This motivates studying which structural and spectral properties are shared by the permanent and determinant, despite the hardness of the permanent.

Analogous to the characteristic polynomial, the \emph{permanental polynomial} of a graph $G$ is
\[
\pi(G,x):=\per(xI-A(G)).
\]
Introduced by Turner~\cite{turner1968generalized}, $\pi(G,x)$ has been studied as a permanent-based
spectral invariant in graph theory and chemical graph theory; see~\cite{kasum1981chemical,merris1981permanental}.
In particular, its coefficients and zeros have been used as structural descriptors and computed
for several chemically motivated families.
Because computing permanents is hard in general, much of the literature focuses on structural
properties and efficient computation for special graph families; see~\cite{GutmanCash2002,HuoLiangBai2006,LiangTongBai2008}.
Computational studies suggest that \emph{copermanental} graphs (nonisomorphic graphs with the same
permanental polynomial) are rare, especially compared with cospectrality for the characteristic
polynomial, so $\pi(G,x)$ may be a strong isomorphism heuristic~\cite{LiuRen2014}.
Related work develops the theory of permanental spectra~\cite{LiuZhang2013properties,pant2025,WuLai2018pernullity}.

\noindent\textbf{Permanental energy.}
Let $\mu_1,\dots,\mu_n$ be the roots of $\pi(G,x)$ (in general $\mu_i\in\mathbb{C}$), and write $|\mu_i|$ for their modulus.
We define the \emph{permanental energy} of $G$ as
\[
\Eper(G)=\sum_{i=1}^n|\mu_i|.
\]
Thus $\Eper(G)$ is a single number that summarizes the permanental roots, in the same spirit
as the usual graph energy.
However, permanental roots are often complex, so many standard eigenvalue techniques do not apply.
Notably, Dehmer \emph{et al.}~\cite{Dehmer2017INS} studied graph descriptors built from the moduli $|\mu_i|$ like
$\sum_i|\mu_i|$ (that is, $\Eper(G)$), $\sum_i|\mu_i|^2$, and more generally $\sum_i|\mu_i|^p$ for various $p$,
and reported that these quantities empirically distinguish networks well in practice.
However, to the best of our knowledge, little general theory is known for such modulus-based permanental descriptors.

Since computing the permanent is $\#P$-complete, exactly computing $\pi(G,x)$ and its roots is infeasible in general; consequently, one should not expect
efficient exact computation of $\Eper(G)$ for arbitrary graphs.
This makes general bounds and extremal characterizations especially valuable, because they provide information about $\Eper(G)$ without requiring explicit root computation.

Despite the hardness of computing $\pi(G,x)$ exactly, the quantity $\Eper(G)$ admits sharp universal extremal behavior. We prove that among all graphs with $m$ edges, the minimum possible permanental energy is exactly $2\sqrt m$, attained only by a star plus
isolated vertices. We prove the general upper bound $\Eper(G)\le n\rho(G)$ and derive basic corollaries (Theorem~\ref{thm:upper-rho} and Corollaries~\ref{cor:upper-m}--\ref{cor:upper-n}). We also give sharper information for several graph families, in Section~\ref{sec:families}.

The remainder of the paper is structured as follows.
Section~\ref{sec:prelim} introduces the necessary definitions, notations, and preliminary results.
Section~\ref{sec:bounds} proves upper and lower bounds on $\Eper(G)$ and characterizes the extremal graphs.
Section~\ref{sec:families} studies $\Eper(G)$ for several graph families.
Section~\ref{sec:open:problems} lists future directions, and the appendix contains proofs deferred from the main text.

\section{Preliminaries}\label{sec:prelim}

For an $n\times n$ matrix $A$ and index sets $I,J\subseteq [n]$, let $A[I,J]$ denote the submatrix
of $A$ formed by rows indexed by $I$ and columns indexed by $J$. In particular, for $S\subseteq [n]$,
we write $A[S,S]$ for the principal submatrix indexed by $S$.

Throughout this paper, graphs are finite, undirected, and simple. 
Let $G$ be a graph with $V(G)=[n]$ and adjacency matrix $A(G)$, and for $S\subseteq[n]$ write
$A(G)[S]:=A(G)[S,S]$. Let $\Delta(G)$ denote the maximum degree of a vertex of $G$.
We write $\sqcup$ for disjoint union of graphs.
Standard graphs are denoted by $K_n$ (complete graph), $K_{1,m}$ (star), $K_1$ (isolated vertex),
$P_n$ (path), and $C_n$ (cycle).
Write $[n]:=\{1,2,\dots,n\}$. We write $\ii$ for the imaginary unit, so $\ii^2=-1$.

We collect below several basic lemmas on permanental polynomials and their roots that will be used throughout the paper.
First recall a standard principal-submatrix expansion for $\per(xI-A)$; see~\cite{minc1984permanents}.

\begin{lemma}\label{lem:principal-expansion}
Let $A\in\mathbb{R}^{n\times n}$ and define its permanental polynomial
\[
\pi(A,x):=\per(xI-A)=\sum_{i=0}^n c_i\,x^{\,n-i}.
\]
Then $c_0=1$, and for each $1\le i\le n$,
\[
c_i = (-1)^i \sum_{\substack{S\subseteq [n]\\ |S|=i}} \per\!\bigl(A[S,S]\bigr).
\]
\end{lemma}

Applying Lemma~\ref{lem:principal-expansion} to $A(G)$ yields the first coefficients of the
permanental polynomial $\pi(G,x)=\per(xI-A(G))$.

\begin{lemma}\label{lem:c1c2c3}
Let $G$ be a simple graph on $n$ vertices with $m$ edges and $t$ triangles. Write
\[
\pi(G,x)=\sum_{k=0}^n c_k\,x^{\,n-k}.
\]
Then $c_0=1$, $c_1=0$, $c_2=m$, and $c_3=-2t$.
\end{lemma}

\begin{proof}
By Lemma~\ref{lem:principal-expansion}, we have
\[
c_k = (-1)^k \sum_{\substack{S\subseteq [n]\\ |S|=k}} \per(A(G)[S]) \qquad (0\le k\le n).
\]
For $k=1$, every $1\times 1$ principal submatrix equals $(0)$, hence $c_1=0$.
For $k=2$, $A(G)[S]=\bigl(\begin{smallmatrix}0&1\\1&0\end{smallmatrix}\bigr)$ if and only if $S$ is an edge,
and then $\per(A(G)[S])=1$, otherwise $\per(A(G)[S])=0$; thus $c_2=m$.
For $k=3$, $\per(A(G)[S])=2$ iff $G[S]\cong K_3$ and otherwise $0$, so
$c_3=-\sum_{|S|=3}\per(A(G)[S])=-2t$.
\end{proof}

We will use the following Newton identities to express coefficients of a polynomial in terms of
power sums of its roots.

\begin{lemma}[{\cite{Marden1966}}]\label{lem:newton12}
Let $p(x)$ be a monic polynomial of degree $n$ with roots $\mu_1,\dots,\mu_n$, and write
\[
p(x)=x^n+a_1x^{n-1}+a_2x^{n-2}+\cdots.
\]
Then
\[
\sum_{i=1}^n \mu_i = -a_1,
\qquad
\sum_{i=1}^n \mu_i^2 = a_1^2-2a_2.
\]
\end{lemma}

\begin{corollary}\label{cor:mu-sums}
Let $G$ be a simple graph with $m$ edges and permanental roots $\mu_1,\dots,\mu_n$. Then
\[
\sum_{i=1}^n \mu_i=0,
\qquad
\sum_{i=1}^n \mu_i^2=-2m.
\]
\end{corollary}

\begin{proof}
Combine Lemmas~\ref{lem:newton12} and~\ref{lem:c1c2c3}.
\end{proof}

We will need the following nonstandard inequality (including its equality conditions) in proving the
lower bound for $\Eper(G)$.

\begin{lemma}\label{lem:l1l2}
Let $z_1,\dots,z_n\in\mathbb{C}$ satisfy $\sum_{i=1}^n z_i=0$. Then
\begin{equation}\label{eq:l1l2}
\Bigl(\sum_{i=1}^n |z_i|\Bigr)^2 \ge 2\sum_{i=1}^n |z_i|^2.
\end{equation}
Moreover, equality holds iff either $z_1=\cdots=z_n=0$, or there exist distinct indices $p\neq q$
and $\alpha\neq 0$ such that $z_p=\alpha$, $z_q=-\alpha$, and $z_i=0$ for all $i\notin\{p,q\}$.
\end{lemma}

\begin{proof}
Let $s:=\sum_{i=1}^n |z_i|$. If $s=0$, then $|z_i|=0$ for all $i$, hence $z_1=\cdots=z_n=0$.
Assume $s>0$. By $\sum_{i=1}^n z_i=0$ we have $z_i=-\sum_{j\ne i} z_j$, and therefore
\[
|z_i|\le \sum_{j\ne i}|z_j|=s-|z_i|.
\]
Thus $|z_i|\le s/2$ for every $i$. Multiplying by $|z_i|$ gives
$|z_i|^2\le (s/2)|z_i|$, and summing over $i$ yields
\begin{equation}\label{eq:l2-upper}
\sum_{i=1}^n |z_i|^2 \le \frac{s}{2}\sum_{i=1}^n |z_i|=\frac{s^2}{2}.
\end{equation}
Rearranging \eqref{eq:l2-upper} gives \eqref{eq:l1l2}.
For the equality case, suppose equality holds in \eqref{eq:l1l2}, equivalently in \eqref{eq:l2-upper}.
Then
\[
\sum_{i=1}^n\Bigl(\frac{s}{2}|z_i|-|z_i|^2\Bigr)
=\frac{s^2}{2}-\sum_{i=1}^n|z_i|^2=0.
\]
Since each summand is nonnegative, we must have $|z_i|^2=(s/2)|z_i|$ for every $i$, hence
$|z_i|\in\{0,s/2\}$.
Since each $|z_i|$ is either $0$ or $s/2$ and their sum is $s$, there exist distinct indices
$p\neq q$ with $|z_p|=|z_q|=s/2$, while $z_i=0$ for all $i\notin\{p,q\}$.
Since $\sum_{i=1}^n z_i=0$, we have $z_p+z_q=0$, hence $z_q=-z_p$. Writing $z_p=\alpha$ gives
$z_q=-\alpha$ with $\alpha\neq 0$.

Conversely, if $z_p=\alpha$, $z_q=-\alpha$, and all other $z_i=0$, then $\sum_{i=1}^n z_i=0$ and
\[
\Bigl(\sum_{i=1}^n |z_i|\Bigr)^2 = (2|\alpha|)^2 = 4|\alpha|^2
= 2\bigl(|\alpha|^2+|\alpha|^2\bigr)=2\sum_{i=1}^n |z_i|^2,
\]
so equality holds in \eqref{eq:l1l2}.
\end{proof}

\section{Bounds and extremal graphs}\label{sec:bounds}

In this section we establish universal lower and upper bounds on $\Eper(G)$ and
characterize the graphs attaining equality.

\subsection{Lower bound}

\begin{theorem}\label{thm:lower-extremal}
Let $G$ be a simple graph with $m$ edges and permanental roots $\mu_1,\dots,\mu_n$. Then
\[
\Eper(G)=\sum_{i=1}^n|\mu_i|\ \ge\ 2\sqrt{m}.
\]
\end{theorem}

\begin{proof}
By Corollary~\ref{cor:mu-sums}, $\sum_{i=1}^n \mu_i=0$ and $\sum_{i=1}^n \mu_i^2=-2m$.
Applying Lemma~\ref{lem:l1l2} to $z_i=\mu_i$ gives
\begin{equation}\label{eq:l1l2-mu}
\Bigl(\sum_{i=1}^n |\mu_i|\Bigr)^2 \ge 2\sum_{i=1}^n |\mu_i|^2.
\end{equation}
Moreover, since $|\mu_i|^2=|\mu_i^2|$,
\begin{equation}\label{eq:l2-lower-mu}
\sum_{i=1}^n |\mu_i|^2=\sum_{i=1}^n |\mu_i^2|
\ \ge\ 
\Bigl|\sum_{i=1}^n \mu_i^2\Bigr|
\ =\ 2m.
\end{equation}
Combining \eqref{eq:l1l2-mu} and \eqref{eq:l2-lower-mu} yields $\Eper(G)^2\ge 4m$, hence
$\Eper(G)\ge 2\sqrt{m}$.
\end{proof}

\begin{theorem}\label{thm:unique-equality}
Let $G$ be a simple graph on $n$ vertices with $m$ edges. Then $\Eper(G)=2\sqrt m$
if and only if $G$ is a star graph on $m+1$ vertices together with $n-m-1$ isolated vertices, that is,
\[
G \cong K_{1,m}\ \sqcup\ (n-m-1)K_1.
\]
\end{theorem}

\begin{proof}
We first prove the forward direction. Assume $\Eper(G)=2\sqrt m$ and let
$\mu_1,\dots,\mu_n$ be the roots of $\pi(G,x)$.
If $m=0$, then $G$ has no edges, so $A(G)=0$ and hence $\pi(G,x)=\per(xI)=x^n$, which implies
$G\cong nK_1$. Thus assume $m\ge 1$.
Equality in Theorem~\ref{thm:lower-extremal} forces equality in \eqref{eq:l1l2-mu}.
Since also $\sum_i \mu_i=0$ by Corollary~\ref{cor:mu-sums}, the equality case of Lemma~\ref{lem:l1l2}
implies that there exist distinct indices $p\ne q$ and $\alpha\ne 0$ such that
\[
\mu_p=\alpha,\qquad \mu_q=-\alpha,\qquad \mu_i=0\ \text{ for all } i\notin\{p,q\}.
\]
Using Corollary~\ref{cor:mu-sums} again, we get
\[
-2m=\sum_{i=1}^n \mu_i^2=\alpha^2+(-\alpha)^2=2\alpha^2,
\]
so $\alpha^2=-m$ and hence $\alpha=\pm \ii\sqrt m$. Therefore
\begin{equation}\label{eq:eqcase-poly}
\pi(G,x)=x^{n-2}(x^2+m)=x^n+m x^{n-2}.
\end{equation}
From~\eqref{eq:eqcase-poly} the coefficient of $x^{n-3}$ is $0$, hence Lemma~\ref{lem:c1c2c3} gives
$t=0$, i.e., $G$ is triangle-free. Also, for every $k\ge 4$ the coefficient of $x^{n-k}$ in
\eqref{eq:eqcase-poly} is $0$. By Lemma~\ref{lem:principal-expansion},
\[
\sum_{\substack{S\subseteq [n]\\ |S|=k}}\per(A(G)[S])=0 \qquad (k\ge 4).
\]
Since each $\per(A(G)[S])\ge 0$, it follows that
\begin{equation}\label{eq:eqcase-principalzero}
\per(A(G)[S])=0 \qquad \text{for every }S\subseteq[n]\text{ with }|S|\ge 4.
\end{equation}
We claim that $G$ has no two vertex-disjoint edges. Indeed, if $\{a,b\}$ and $\{c,d\}$ are disjoint
edges and $S=\{a,b,c,d\}$, then the permutation $(a\,b)(c\,d)$ contributes $1$ to $\per(A(G)[S])$,
contradicting~\eqref{eq:eqcase-principalzero}. Hence any two edges of $G$ share a common endpoint.

Let $c$ be a vertex of maximum degree $\Delta\ge 1$. If $\Delta=1$, then $G$ has at most one edge,
so $G\cong K_{1,1}\sqcup (n-2)K_1$, which matches the claim for $m=1$.
Assume $\Delta\ge 2$ and choose two distinct neighbors $u$ and $v$ of $c$.
If there were an edge $e$ not incident to $c$, then $e$ must meet both $\{c,u\}$ and $\{c,v\}$.
Since $e$ avoids $c$, this forces $e=\{u,v\}$, creating a triangle $cuv$, a contradiction.
Therefore every edge of $G$ is incident to $c$, so the non-isolated part of $G$ is a star with
center $c$ and $m$ edges. Hence $G\cong K_{1,m}\sqcup (n-m-1)K_1$.

For the converse direction, if $G\cong K_{1,m}\sqcup (n-m-1)K_1$, then
$\pi(G,x)=x^{n-2}(x^2+m)$ and the permanental roots are $0$ (with multiplicity $n-2$) and
$\pm \ii \sqrt m$, so $\Eper(G)=2\sqrt m$.
\end{proof}

\begin{corollary}\label{cor:connected-min}
If $G$ is connected on $n\ge 2$ vertices, then
\[
\Eper(G)\ge 2\sqrt{n-1},
\]
with equality iff $G\cong K_{1,n-1}$.
\end{corollary}

\begin{proof}
Let $G$ have $m$ edges. Since $G$ is connected, $m\ge n-1$, so Theorem~\ref{thm:lower-extremal} gives
\[
\Eper(G)\ge 2\sqrt m\ge 2\sqrt{n-1}.
\]
If equality holds, then $m=n-1$ and Theorem~\ref{thm:unique-equality} forces $G$ to be a star, hence
$G\cong K_{1,n-1}$.
Conversely, if $G\cong K_{1,n-1}$, then $\pi(G,x)=x^{n-2}(x^2+n-1)$ and hence
$\Eper(G)=2\sqrt{n-1}$.
\end{proof}

\subsection{Upper bounds}

For a square matrix $M$, let $\rho(M)$ be its spectral radius, that is, the largest absolute value among its eigenvalues.
In particular, for a graph $G$, we write $\rho(G):=\rho(A(G))$.
We use a standard bound of Brenner and Brualdi~\cite{BrennerBrualdi1967} on the roots of $\per(xI-A)$ to upper bound
$\Eper(G)$.

\begin{lemma}[Brenner--Brualdi {\cite{BrennerBrualdi1967}}]\label{lem:BB-disk}
Let $A\ge 0$ be an $n\times n$ matrix. Every root $\mu$ of $\per(xI-A)$ satisfies $|\mu|\le \rho(A)$.
\end{lemma}

\begin{theorem}\label{thm:upper-rho}
Let $G$ be a simple graph on $n$ vertices with spectral radius $\rho=\rho(G)$. Then
\[
\Eper(G)\le n\rho.
\]
\end{theorem}

\begin{proof}
Let $\mu_1,\dots,\mu_n$ be the roots of $\pi(G,x)=\per(xI-A(G))$.
By Lemma~\ref{lem:BB-disk} (applied to $A(G)\ge 0$), we have $|\mu_i|\le \rho$ for all $i$.
Summing gives $\Eper(G)\le n\rho$.
\end{proof}

Thus, any upper bound on $\rho(G)$ immediately yields a corresponding upper bound on $\Eper(G)$
via Theorem~\ref{thm:upper-rho}.

\begin{corollary}\label{cor:upper-m}
Let $G$ be a simple graph on $n$ vertices with $m$ edges. Then
\[
\Eper(G)\le \frac{n(\sqrt{8m+1}-1)}{2}.
\]
\end{corollary}

\begin{proof}
Combine Theorem~\ref{thm:upper-rho} with Stanley's bound
$\rho(G)\le (\sqrt{8m+1}-1)/2$~\cite{Stanley1987}.
\end{proof}

\begin{corollary}\label{cor:upper-n}
For every simple graph $G$ on $n\ge 2$ vertices,
\[
\Eper(G)\le n\Delta(G)\le n(n-1).
\] 
\end{corollary}

\begin{proof}
Use $\rho(G)\le \Delta(G)\le n-1$ in Theorem~\ref{thm:upper-rho}.
\end{proof}

\begin{remark}[Tightness of the upper bound]
Theorem~\ref{thm:upper-rho} is usually not tight, since it upper bounds all permanental roots by
$\rho(G)$. For instance, $\Eper(K_{1,n-1})=2\sqrt{n-1}$, while
Theorem~\ref{thm:upper-rho} gives $\Eper(K_{1,n-1})\le n\sqrt{n-1}$.
\end{remark}

\paragraph{Comparison with adjacency energy.}
It is well known that $2\sqrt m\le \E(G)\le \sqrt{2mn}$ and both bounds are tight ; see~\cite{Gutman1978}.
Theorem~\ref{thm:lower-extremal} shows that the same sharp lower bound $2\sqrt m$ holds for
$\Eper(G)$, but the upper bound $\sqrt{2mn}$ does not extend to $\Eper(G)$:
for $G=K_3$, $\pi(K_3,x)=x^3+3x-2$ has roots
$\mu_1\approx 0.5961$ and $\mu_{2,3}\approx -0.2980\pm 1.8073\,\ii$, hence
$\Eper(K_3)\approx 4.2596>\sqrt{18}\approx 4.2426$.

This motivates the extremal problem of maximizing $\Eper(G)$ over all graphs on $n$
vertices. Based on computations for all graphs up to $n\le 10$, we observe that the complete graph
is the maximizer; see Section~\ref{sec:complete-graphs}.

\section{Special families}\label{sec:families}
In this section we discuss graph families for which $\Eper(G)$ admits sharper information than the
universal bounds of Section~\ref{sec:bounds}.

\subsection{Bipartite families}\label{sec:families:bip}

For some bipartite graph classes, the permanental polynomial can be written in terms of the characteristic
polynomial of an associated matrix. This can force $\Eper(G)$ to coincide with the usual
adjacency energy $\E(G)$ and can also yield stronger upper bounds for such classes.

If $G$ is bipartite and has no cycle whose length is divisible by $4$, then Borowiecki~\cite{borowiecki1985spectrum}
proved the following relation between the permanental and characteristic polynomials.

\begin{theorem}\label{thm:c4kfree-perspectrum}
If $G$ is bipartite and contains no cycle of length $4k$ for any $k\ge 1$, then
\[
\pi(G,x)=\ii^{-n}\chi(G,\ii x).
\]
Equivalently, the permanental roots satisfy
$\{\mu_1,\dots,\mu_n\}=\{-\ii\lambda_1,\dots,-\ii\lambda_n\}$,
where $\lambda_1,\dots,\lambda_n$ are the adjacency eigenvalues (as a multiset).
\end{theorem}

\begin{corollary}\label{cor:c4kfree-energy}
If $G$ is bipartite and contains no cycle of length $4k$ for any $k\ge 1$, then $\Eper(G)=\E(G)$.
\end{corollary}

\begin{proof}
By Theorem~\ref{thm:c4kfree-perspectrum}, the permanental roots are $\{-\ii\lambda_j\}_{j=1}^n$
(with multiplicities), hence $|\mu_j|=|\lambda_j|$ and $\Eper(G)=\E(G)$.
\end{proof}

An \emph{orientation} assigns a direction to each edge.  Since $G$ is bipartite, every cycle of $G$
has even length.  We say that a cycle is \emph{oddly oriented} in an orientation $\sigma$ if,
when traversing the cycle, an odd number of its edges are directed along the traversal.
Given an orientation $G^\sigma$, let $S(G^\sigma)$ be the matrix with
$S_{uv}=1$ if $u\to v$ is an arc of $G^\sigma$, $S_{uv}=-1$ if $v\to u$ is an arc, and
$S_{uv}=0$ otherwise.
Zhang and Li~\cite{zhang2014note} showed that if $G$ admits an orientation in which every cycle is
oddly oriented, then the permanental polynomial has an exact determinantal description.

\begin{theorem}[{\cite{zhang2014note}}]\label{thm:skew-det-model}
Let $G$ be bipartite. If $G$ admits an orientation $\sigma$ in which every cycle is oddly oriented, then
\[
\pi(G,x)=\per(xI-A(G))=\det\bigl(xI-S(G^\sigma)\bigr).
\]
\end{theorem}

This gives a McClelland-type upper bound for this class.

\begin{theorem}\label{thm:det-mcclelland}
Let $G$ be a bipartite graph on $n$ vertices with $m$ edges. Suppose that $G$ admits an orientation
$\sigma$ in which every cycle of $G$ is oddly oriented. Then
\[
\Eper(G)\le \sqrt{2mn}.
\]
\end{theorem}

\begin{proof}
Write $S:=S(G^\sigma)$. By Theorem~\ref{thm:skew-det-model}, $\pi(G,x)=\det(xI-S)$.
Set $H:=-\ii S$. Since $S^{\mathsf T}=-S$ and $S$ is real, $H$ is Hermitian and
\[
\pi(G,x)=\det(xI-\ii H).
\]
Thus the roots of $\pi(G,x)$ are $\ii\lambda_1,\dots,\ii\lambda_n$, where
$\lambda_1,\dots,\lambda_n\in\mathbb{R}$ are the eigenvalues of $H$. Hence
\[
\Eper(G)=\sum_{j=1}^n|\lambda_j|.
\]
By Cauchy--Schwarz,
\[
\Eper(G)^2 \le n\sum_{j=1}^n \lambda_j^2 = n\,\tr(H^2).
\]
Moreover,
\[
\tr(H^2)=\sum_{u,v}|h_{uv}|^2=\sum_{u,v}|s_{uv}|^2=2m,
\]
because $S$ has exactly $2m$ off-diagonal entries equal to $\pm 1$.
Therefore $\Eper(G)^2\le 2mn$, that is, $\Eper(G)\le \sqrt{2mn}$.
\end{proof}

\subsection{Forests and tree extremals}\label{sec:forests}

Forests are bipartite and contain no cycles, hence satisfy Theorem~\ref{thm:c4kfree-perspectrum}.
Therefore, for every forest $F$ one has $\Eper(F)=\E(F)$.

\begin{corollary}\label{cor:tree-extremals}
Among all trees $T$ on $n$ vertices,
\[
\Eper(K_{1,n-1}) \le \Eper(T) \le \Eper(P_n),
\]
with equality on the left if and only if  $T\cong K_{1,n-1}$ and equality on the right if and only if $T\cong P_n$.
\end{corollary}

\begin{proof}
Since every tree is a forest, we have $\Eper(T)=\E(T)$.
Apply the classical extremal theorem for adjacency energy on trees~\cite{Gutman1978}.
\end{proof}
\subsection{Cycles}\label{sec:cycles}

Cycles can be handled explicitly. The next theorem gives $\Eper(C_n)$, with different behavior for even and odd $n$.

\begin{theorem}\label{thm:cycle-energy}
Let $n\ge 3$.
\begin{enumerate}
\item If $n$ is even, then
\[
\Eper(C_n)=\frac{4}{\sin(\pi/n)}.
\]
\item If $n$ is odd, then
\[
\lim_{\substack{n\to\infty\\ n\ \mathrm{odd}}}\frac{\Eper(C_n)}{n}=\frac{4}{\pi}.
\]
\end{enumerate}
\end{theorem}

\begin{proof}
See Appendix~\ref{app:cycles}.
\end{proof}

\subsection{Complete graphs}\label{sec:complete-graphs}

The general bound $\Eper(G)\le n\rho(G)$ suggests that graphs with large spectral radius may
also have large permanental energy. In particular, the complete graph $K_n$ is a natural candidate.
Let $D_n$ denote the number of derangements of $[n]$, that is, permutations $\sigma\in S_n$ with
$\sigma(i)\neq i$ for all $i\in[n]$.

\begin{lemma}\label{lem:Kn-derangement}
For $n\ge 1$, $\per(A(K_n))=D_n$.
\end{lemma}

\begin{proof}
In $A(K_n)$ the diagonal entries are $0$ and all off-diagonal entries are $1$. Hence a permutation
$\sigma$ contributes $1$ to $\per(A(K_n))$ iff $\sigma(i)\neq i$ for all $i$, that is, iff $\sigma$ is a
derangement.
\end{proof}

\begin{lemma}\label{lem:Kn-derangement-lower}
For every $n\ge 1$,
\[
\Eper(K_n)\ \ge\ n\,D_n^{1/n}.
\]
\end{lemma}

\begin{proof}
Let $\mu_1,\dots,\mu_n$ be the roots of the monic polynomial $\pi(K_n,x)$. Since
$\pi(K_n,0)=\per(-A(K_n))=(-1)^n\per(A(K_n))$, we have
\[
\prod_{i=1}^n \mu_i = (-1)^n\pi(K_n,0)=(-1)^n(-1)^n\per(A(K_n))=\per(A(K_n))=D_n
\]
by Lemma~\ref{lem:Kn-derangement}. Taking absolute values gives $\prod_{i=1}^n |\mu_i|=D_n$.
Applying AM--GM to $|\mu_1|,\dots,|\mu_n|$ yields
\[
\frac{1}{n}\sum_{i=1}^n|\mu_i|\ \ge\ \Bigl(\prod_{i=1}^n|\mu_i|\Bigr)^{1/n}=D_n^{1/n},
\]
which implies $\Eper(K_n)=\sum_{i=1}^n|\mu_i|\ge nD_n^{1/n}$.
\end{proof}

\begin{corollary}\label{cor:Kn-growth-asymp}
As $n\to\infty$,
\[
\Eper(K_n)\ \ge\ \Bigl(\frac1e-o(1)\Bigr)n^2.
\]
In particular, $\Eper(K_n)=\Omega(n^2)$.
\end{corollary}

\begin{proof}
By Lemma~\ref{lem:Kn-derangement-lower}, $\Eper(K_n)\ge n\,D_n^{1/n}$.
It is standard that $D_n=\frac{n!}{e}+O(1)$ (see~\cite{GrahamKnuthPatashnik1994}), hence
$D_n^{1/n}=\bigl(\frac{n!}{e}\bigr)^{1/n}\,(1+o(1))$.
By Stirling's formula, $(n!)^{1/n}=\bigl(1+o(1)\bigr)\frac{n}{e}$, so
$D_n^{1/n}=\bigl(1+o(1)\bigr)\frac{n}{e}$.
Therefore $\Eper(K_n)\ge \bigl(\frac1e-o(1)\bigr)n^2$.
\end{proof}

As a simple consequence, $K_n$ already dominates every $n$-vertex graph whose spectral radius is at most $D_n^{1/n}$.
\begin{theorem}\label{thm:Kn-dominates-small-rho}
Let $G$ be a simple graph on $n$ vertices. If $\rho(G)\le D_n^{1/n}$, then
\[
\Eper(G)\le \Eper(K_n).
\]
\end{theorem}

\begin{proof}
By Theorem~\ref{thm:upper-rho}, $\Eper(G)\le n\rho(G)\le nD_n^{1/n}$.
By Lemma~\ref{lem:Kn-derangement-lower}, $\Eper(K_n)\ge nD_n^{1/n}$.
Thus $\Eper(G)\le \Eper(K_n)$.
\end{proof}

\begin{remark}
Exhaustive enumeration for all graphs up to $n\le 10$ indicates that
$K_n$ is the unique maximizer of $\Eper(G)$ among $n$-vertex graphs. Moreover, for
$3\le n\le 8$, the best non-complete graph is $K_n$ with one edge deleted and in the dense
regime $m=\binom{n}{2}-k$ the maximizer appears to be the clique with a matching of size $k$ deleted.
\end{remark}
\section{Future Directions}\label{sec:open:problems}

Several interesting questions remain open. On the algorithmic side, it would be interesting to determine
the complexity of computing $\Eper(G)$ exactly, and more realistically to understand the
complexity of approximating $\Eper(G)$ within additive error $\varepsilon n$ or relative
error $1\pm\varepsilon$. On the extremal side, our general upper bound $\Eper(G)\le n\rho(G)$
is often not tight; finding universal upper bounds that improve on this in nontrivial regimes seems particularly worthwhile. Structurally, beyond the bipartite classes
in Section~\ref{sec:families}, it remains to characterize broad graph families for which
$\Eper(G)=\E(G)$, or to identify clear obstructions to equality. Finally, developing general
constraints on the location and multiplicity pattern of permanental roots (forced zeros, symmetry,
stability phenomena, etc.) could provide new ways to control $\Eper(G)$ without explicitly
computing permanental roots.

\bibliographystyle{plain}
\bibliography{references}

\appendix

\section{Proof of Theorem~\ref{thm:cycle-energy}}\label{app:cycles}

This appendix proves Theorem~\ref{thm:cycle-energy}. We use the following explicit expression for
$\pi(C_n,x)$ due to Li and Zhang.

\begin{theorem}[{\cite{LiZhang2012Certain}}]\label{thm:cycle-poly-LZ}
For $n\ge 3$,
\[
\pi(C_n,x)=
\begin{cases}
\displaystyle \prod_{t=1}^{n}\Bigl(x+2\ii\sin\frac{(2t-1)\pi}{n}\Bigr), & \text{if $n$ is even},\\[1.2ex]
\displaystyle \prod_{t=1}^{n}\Bigl(x+2\ii\cos\frac{t\pi}{n+1}\Bigr)
\;+\;
\prod_{t=1}^{n-2}\Bigl(x+2\ii\cos\frac{t\pi}{n-1}\Bigr)
\;-\;2, & \text{if $n$ is odd}.
\end{cases}
\]
\end{theorem}

We also use the following standard trigonometric identity.
\begin{lemma}[{\cite{GradshteynRyzhik}}]\label{lem:odd-sine-sum}
For $m\ge 1$ and $\theta\not\equiv 0 \pmod{\pi}$,
\[
\sum_{j=1}^{m}\sin\bigl((2j-1)\theta\bigr)=\frac{\sin^2(m\theta)}{\sin\theta}.
\]
\end{lemma}

\subsection{Even cycles}
\begin{corollary}\label{cor:even-cycle-energy-app}
If $n\ge 4$ is even, then
\[
\Eper(C_n)=\frac{4}{\sin(\pi/n)}.
\]
\end{corollary}

\begin{proof}
By Theorem~\ref{thm:cycle-poly-LZ} (even case), the roots are
$\mu_t=-2\ii\sin\frac{(2t-1)\pi}{n}$ for $t=1,\dots,n$. Hence
\[
\Eper(C_n)=\sum_{t=1}^n|\mu_t|
=2\sum_{t=1}^n\Bigl|\sin\frac{(2t-1)\pi}{n}\Bigr|
=4\sum_{t=1}^{n/2}\sin\frac{(2t-1)\pi}{n},
\]
using $\sin(\pi-\alpha)=\sin\alpha$. Apply Lemma~\ref{lem:odd-sine-sum} with $m=n/2$ and $\theta=\pi/n$:
\[
\sum_{t=1}^{n/2}\sin\frac{(2t-1)\pi}{n}
=\frac{\sin^2(\pi/2)}{\sin(\pi/n)}=\frac{1}{\sin(\pi/n)}.
\]
Thus $\Eper(C_n)=4/\sin(\pi/n)$.
\end{proof}

\subsection{Odd cycles}
For odd $n$, we use a closed form for $\pi(C_n,x)$ in terms of
\[
\lambda_\pm=\frac{x\pm\sqrt{x^2+4}}{2},
\]
and then extract an explicit root parametrization to deduce the asymptotic behavior of
$\Eper(C_n)$.

\begin{lemma}[{\cite{kasum1981chemical}}]\label{lem:cycle-lambda-form}
For $n\ge 3$,
\[
\pi(C_n,x)=\lambda_+^{\,n}+\lambda_-^{\,n}+2(-1)^n.
\]
\end{lemma}

\begin{lemma}\label{lem:odd-cycle-roots-app}
Let $n\ge 3$ be odd and set $a:=\log(1+\sqrt2)$. Then the roots of $\pi(C_n,x)$ are
\[
\mu_k=2\sinh\!\Bigl(\frac{a+2\pi\ii k}{n}\Bigr),\qquad k=0,1,\dots,n-1.
\]
\end{lemma}

\begin{proof}
Let $n$ be odd and write $\lambda:=\lambda_+$. Then $\lambda_-=-1/\lambda$ and $(-1)^n=-1$.
By Lemma~\ref{lem:cycle-lambda-form},
\[
0=\pi(C_n,x)=\lambda^{n}-\lambda^{-n}-2.
\]
Let $y:=\lambda^n$. Then $y-y^{-1}=2$, i.e., $y^2-2y-1=0$, so $y=1\pm\sqrt2$.
Take $y=1+\sqrt2=e^{a}$ and write
\[
\lambda=\exp\!\Bigl(\frac{a+2\pi\ii k}{n}\Bigr),\qquad k=0,1,\dots,n-1.
\]
Finally,
\[
x=\lambda-\frac{1}{\lambda}
=\exp\!\Bigl(\frac{a+2\pi\ii k}{n}\Bigr)-\exp\!\Bigl(-\frac{a+2\pi\ii k}{n}\Bigr)
=2\sinh\!\Bigl(\frac{a+2\pi\ii k}{n}\Bigr),
\]
which gives the claimed roots.
\end{proof}

\begin{corollary}\label{cor:odd-cycle-asymp-app}
If $n$ is odd, then
\[
\Eper(C_n)=\frac{4}{\pi}n+o(n)\qquad (n\to\infty,\ n\ \text{odd}).
\]
\end{corollary}

\begin{proof}
Let $u:=a/n$ and $\theta_k:=2\pi k/n$. By Lemma~\ref{lem:odd-cycle-roots-app},
\[
\mu_k=2\sinh(u+\ii\theta_k).
\]
Using $|\sinh(u+\ii\theta)|^2=\sinh^2u+\sin^2\theta$, we obtain
\[
\frac{1}{n}\Eper(C_n)
=\frac{2}{n}\sum_{k=0}^{n-1}\sqrt{\sinh^2(a/n)+\sin^2\theta_k}.
\]
Since $\sinh(a/n)=O(1/n)$, uniformly in $k$ we have
\[
\sqrt{\sinh^2(a/n)+\sin^2\theta_k}=|\sin\theta_k|+o(1).
\]
Therefore
\[
\frac{1}{n}\Eper(C_n)=\frac{2}{n}\sum_{k=0}^{n-1}|\sin\theta_k|+o(1),
\]
and the sum is a Riemann sum for $\frac{1}{2\pi}\int_0^{2\pi}|\sin t|\,dt=\frac{2}{\pi}$, giving the limit $4/\pi$.
\end{proof}

\end{document}